\documentclass[10pt,twoside]{article}
\usepackage{amsmath,amsfonts,amssymb,amsthm,graphicx,amscd}

\newtheorem{theorem}{Theorem}
\newtheorem{lemma}{Lemma}
\newtheorem{definition}{Definition}

\begin{document}

%%%%%%%%%%%%%%%%%%%%%%%%%%%%%%%%%%%%%%%
\title{Hyperbolic space groups and edge conditions for their domains}

\author{Milica Stojanovi\' c
\\ University of Belgrade, Faculty of Organizational Sciences,\\
154 Jove Ili\' ca Street, 11040 Belgrade, Serbia\\ 
E-mail: milica.stojanovic@fon.bg.ac.rs\\
https://orcid.org/0000-0002-2517-2189\\
}
\date{}

\maketitle

\pagestyle{myheadings} \markboth{\rm M. Stojanovi\'{c}}{\rm Hyperbolic space groups and edge conditions of its domain}


\begin{abstract}

Looking to the fundamental domains of space groups we can investigate in which space they can be realized. If this space is hyperbolic, then the corresponding space group is also hyperbolic. In addition to the usual methods for investigating space of realization, the symmetries of the fundamental polyhedron can give new restricted conditions, here called edge conditions. 

The aim of the research is to find out in which cases simplicial fundamental domains are hyperbolic with vertices out of the absolute. For this reason, edge conditions for simplicial fundamental domains belonging to Family F12 by the notation of E. Moln\' ar et all in 2006, are considered.  
\end{abstract}

%%%%%%%%%%%%%%%%%%%%%%%%%%%%%%%%%%%%%%%%%%%%%%%%%%%%%%%%%%%%%%%%%%%%%%%

%%%%%%%%%%%%%%%%%%%%%%%%%%%%%%%%%%%%%%%%%%%%%%%%%%%%%%%%%%%%%%%%%%%%%%%

\section{Introduction} \label{sec:1}

%%%%%%%%%%%%%%%%%%%%%%%%%%%%%%%%%%

Hyperbolic space groups are isometry groups that act discontinuously on hyperbolic 3-space with a compact fundamental domain. One way to classify them is to look for the fundamental domains of these groups. Face pairing identifications on a given polyhedron give us generators and relations for the space group by the algorithmically generalized Poincar\'{e} Theorem \cite{BM}. 

The simplest fundamental domains are simplices, as well as their integer parts by inner symmetries, and truncated simplices (trunc-simplices) by polar planes of vertices when they lie out of the absolute. There are 64 combinatorially different face pairings of fundamental simplices \cite{Z}, \cite{MP94}, furthermore 35 solid transitive non-fundamental simplex identifications \cite{MP94}. I. K. Zhuk \cite{Z} classified Euclidean and hyperbolic fundamental simplices of finite volume up to congruence. The algorithmic procedure for classification was given by E. Moln\'{a}r and I. Prok \cite{MP88}, while in \cite{MP94}, \cite{MPS97} and \cite{MPS06} the authors summarized all these results, arranging the identified simplices into 32 families. Each of them is characterized by the so-called maximal series of simplex tilings. Besides spherical, Euclidean and hyperbolic realizations there exist also other metric realizations in 3-dimensional simply connected homogeneous Riemannian spaces, moreover, metrically non-realizable topological simplex tilings occur as well \cite{M97}. Some complete cases of supergroups with fundamental trunc-simplices
are discussed in \cite{M90}, \cite{S93}, \cite{S95}, \cite{S97}, \cite{S11}. As a classification, in \cite{S19}, supergroups of groups with fundamental trunc-simplices are given, while in \cite{MSS23} supergroups for groups with fundamental domains that are integer parts of trunc-simplices belonging to families F1 - F4 are considered.

Each identified simplex, considered in this paper, has two equivalence classes for edges, each with 3 edges and one for vertices. We obtain such a division into equivalence classes in 4 cases of face pairings. According to the notation in \cite{MPS06}, these are simplices $T_{19}$, $T_{46}$, $T_{59}$ in the family F12 and simplex $T_{31}$ in the family F27. In \cite{MPS06}, the space of realization of these simplices is also considered using the conditions provided by the Coxeter-Schl\"{a}fli matrix (\cite{V93}). The corresponding groups with fundamental trunc-simplexes are given in \cite{S11}. Here we additionally take into account the conditions obtained by the symmetries of the fundamental simplices, which we call edge conditions here. After examining similar conditions in \cite{S93}, \cite{S97}, cases of simplices realized in the spaces $Nil$ and $\tilde{SL_2} (\mathbb{R})$ were discovered.

This problem was published in 1996, in the author's PhD Thesis \cite{S95}, but in Serbian. As this research becomes more and more relevant, because research on hyperbolic space groups has a newer applications in today's material sciences (e.g. fullerenes, nanomaterials), we repeat it here in an adapted and improved way. Extremal problems related to hyperbolic groups as optimal ball packing and covering, (e.g. \cite{ESz23}, \cite{MSS23} \cite{MS24}, \cite{Sz17}) select these infinite series to choose models for optimal real materials.

In Section \ref{sec:2}, the projective metric is introduced. In Subsection \ref{sec:3-1} the necessary conditions for simplices are given, while in Subsection \ref{sec:3-2} the main Theorem \ref{th:1} is formulated. The proof of Theorem \ref{th:1} is given in Section \ref{sec:4} through Lemmas \ref{l:2} - \ref{l:8}.

%%%%%%%%%%%%%%%%%%%%%%%%%%%%%%%

\section{Projective metric, spherical and hyperbolic space}
\label{sec:2}

Consider the 4-dimensional real vector space $V^4$ and its dual space $\mathcal{V}^*_4$ of linear forms. The projective 3-space $P^3(V^4, \mathcal{V}^*_4)$ can be introduced in the usual way. The 1-dimensional subspaces of $V^4$ (or 3-subspaces of $\mathcal{V}^*_4$) represent the points of $P^3$, while the 1-subspaces of $\mathcal{V}^*_4$ (or the 3-subspaces of $V^4$) represent the planes of $P^3$. The point $X(x)$ and the plane $\alpha (a)$ are incident iff $xa = 0$, i.e. the value of the linear form $a$ on the vector $x$ is equal to zero $(x \in V^4 \setminus \{ 0 \} , a \in \mathcal{V}^*_4 \setminus \{ 0 \} )$. The straight lines of $P^3$ are characterized by 2-subspaces of $V^4$ or of $\mathcal{V}^*_4$, respectively. If $\{ e_i \}$ is a basis for $V^4$ and $\{ e^j \}$ is its dual basis on $\mathcal{V}^*_4$, i.e. $e_i e^j = \delta _i^j$ (the Kroneker symbol), then the form $a = e^j a_j$ takes the value $x a = x^i a_i$ on the vector $x = x^i e_i$. We use the summation convention for the same upper and lower indices.

We can introduce projective metric in $P^3$ by giving a bilinear form 
$$\left\langle ; \right\rangle : \mathcal{V}^*_4 \times \mathcal{V}^*_4 \rightarrow \mathbb{R}, \ \ 
\left\langle b^i u_i; b^j v_j \right\rangle = u_i b^{ij} v_j$$
where $\left( \left\langle b^i; b^j \right\rangle \right) = \left( b^{ij} \right) = \mathbb{B}$ is the Coxeter-Schl\"{a}fli matrix, and the basis $\left\{ b^i \right\}$ in $\mathcal{V}^*_4$ represents the planes containing the simplex faces opposite to the vertices $A_i$, respectively. The vectors $a_j$ of the dual basis $\left\{ a_j \right\}$ in $V^4$, defined by $a_j b^i = \delta _j^i$, represent the vertices $A_j$ of the simplex. The induced bilinear form
$$\left\langle ; \right\rangle : V^4 \times V^4 \rightarrow \mathbb{R}, \ \ 
\left\langle x^i a_i; y^j a_j \right\rangle = x^i a_{ij} y^j$$
is defined by the matrix $\left( \left\langle a_i; a_j \right\rangle \right) = \left( a_{ij} \right) = \mathbb{A}$ inverse to $\mathbb{B}$.

We assume that the bilinear form $\left\langle ; \right\rangle$ is either of the signature $(+, +, +, -)$ which characterizes the hyperbolic metric, or $(+, +, +, +)$ which characterizes the elliptic (spherical) metric. The signature $(+, +, +, 0)$ would describe the Euclidean geometry, which will not appear in our consideration. We can find this signature by finding the eigenvalues of the matrix $\mathbb{B}$.

It is well-known that the bilinear form induces the distance and the angle measure in 3-space. Let $X(x)$ and $Y(y)$ be two points in the projective space $P^3$. Then their distance $d(x,y)$ is determined by
\begin{equation} \label{distance}
\cos \left(d(x,y)\right) = \frac{\left\langle x ; y \right\rangle}{\sqrt{\left\langle x ; x \right\rangle   \left\langle y ; y \right\rangle}} \ \ \mathrm{and} \ \ \cosh \left(d(x,y)\right) = \frac{ - \left\langle x ; y \right\rangle}{\sqrt{\left\langle x ; x \right\rangle   \left\langle y ; y \right\rangle}}
\end{equation}
for the elliptic and hyperbolic cases, respectively. We shall also use the following

\begin{lemma} \label{l:1} For any $(r+1)$-minor determinant of a regular matrix $\left (a_{ij} \right)$ and complementary $\left( n - r \right)$-minor of its inverse $\left (b^{ij} \right)$ holds the following equality
$$
\left|
\begin{array}{ccc}
	a_{i_0 j_0} & \ldots & a_{i_0 j_r} \\
	\ldots &  & \ldots \\
	a_{i_r j_0} & \ldots & a_{i_r j_r}
\end{array}
\right|
= \det (a_{ij})
\left|
\begin{array}{ccc}
	b^{i_{r+1} j_{r+1}} & \ldots & b^{i_{r+1} j_n} \\
	\ldots &  & \ldots \\
	b^{i_n j_{r+1}} & \ldots & b^{i_n j_n}
\end{array}
\right|
\sigma.$$
Here $\sigma = \mathrm{sign} \left( i_0, \ldots , i_r, i_{r+1}, i_n  \right) \cdot \mathrm{sign} \left( j_0, \ldots , j_r, j_{r+1}, j_n  \right)$ denotes the sign product of the corresponding permutations of the elements $0,1,\ldots , n$.
\end{lemma}

%%%%%%%%%%%%%%
For hyperbolic simplices it is interesting to investigate the cases when the vertices are proper, or they lie on the absolute, or out of the absolute. Thus, we need  a submatrix $\hat{B}_{ii}$ of $\mathbb{B}$ corresponding to the vertex $A_i$ obtained by excluding the $i$-th row and $i$-th column. Since we are interested here in the vertices $A_i$ out of the absolute, we need the signature $(+, +, -)$ for $\hat{B}_{ii}$.

%%%%%%%%%%%%%%%%%%%%%%%

\section{Necessary conditions for simplices and the main Theorem}
\label{sec:3}

\subsection{Necessary conditions for simplices}
\label{sec:3-1}

Here we consider the space of realization for the simplices $T_{19}$, $T_{46}$, $T_{59}$ in the family F12 and the simplex $T_{31}$ in the family F27, following the notations in \cite{MPS06}. (Fig. \ref{fig:1}).

The Coxeter-Schl\"{a}fli matrix for all simplices considered here is the same

$$\mathbb{B} = \left[
\begin{array}{cccc}
1 & -\cos (\beta _1) & -\cos (\beta _2) & -\cos (\alpha _1) \\
-\cos (\beta _1) & 1 & -\cos (\alpha _1) & -\cos (\alpha _2) \\
-\cos (\beta _2) & -\cos (\alpha _1) & 1 & -\cos (\beta _1) \\
-\cos (\alpha _1) & -\cos (\alpha _2) & -\cos (\beta _1) & 1 
\end{array}
\right], \ \mathrm{where} $$
%\end{eqnarray} 
\begin{eqnarray} \label{AB}
2 \alpha _1 + \alpha _2 = \frac{2 \pi}{a}, \ \ \ 2 \beta _1 + \beta _2 = \frac{2 \pi}{b}.
\end{eqnarray} 

The following applies to parameters $a$ and $b$: for simplex $T_{19}$, $a \in 2 \mathbb{N}$, $b \in 2 \mathbb{N}$; for $T_{46}$, $a \in \mathbb{N}$, $b \in 2 \mathbb{N}$; for $T_{59}$, $a \in \mathbb{N}$, $b \in \mathbb{N}$; for $T_{31}$, $a \in 2 \mathbb{N}$, $b \in 4 \mathbb{N}$. Here we shall consider the general case for parameters $a$ and $b$, when they can take any natural value. Then we shall consider the space in which the general simplex $T$ is realized (as a model for any of the above). Since there is a duality between '$a$ lines' and '$b$ lines', we can only consider the cases $b \geq a$. Also, if $a = b$, the simplex has more symmetries and belongs to Family 1, according to the notation in \cite{MPS06} and will not be considered here.

\begin{figure}[htbp]

	\centering
		\includegraphics[width=0.98\textwidth]{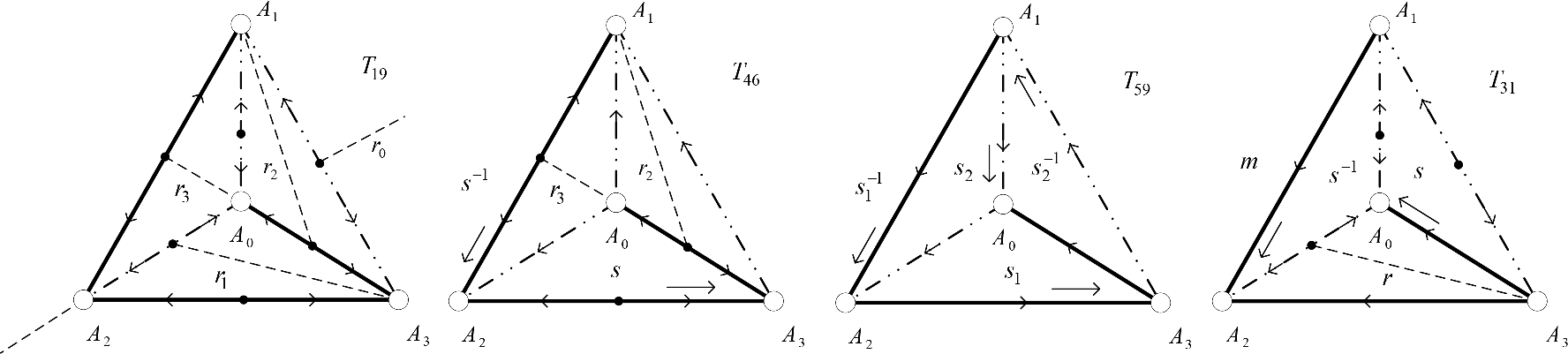}
		\caption{\label{fig:1} \small The simplices $T_{19}$, $T_{46}$, $T_{59}$ and $T_{31}$}
	
\end{figure}

Based on the Coxeter-Schl\"{a}fli matrix, for the general case of the simplex $T$ considered here, it was established in \cite{MPS06} that if $a = 1$, the simplex is realized in $\mathbb{S}^3$, if $(a, b) = (2, 2)$ the simplex is in $\mathbb{H}^3$, with ideal vertices, while in other cases it is hyperbolic with outer vertices. In \cite{Z} was also established that $a = 1$ leads to $\mathbb{S}^3$. Therefore, for our consideration there remain the cases $b > a \geq 2$, where based on previous results (\cite{MPS06}, \cite{Z}, \cite{S95}) we can assume that $\Delta = \det (\mathbb{B}) < 0$ and also $B_{ii} = \det (\hat{B}_{ii}) < 0$, $i = 0, 1, 2, 3$.

Other important properties about the considered simplices are given in \cite{MPS06} and \cite{S11}.

\vspace{5mm}

Here we note that in addition to the previously considered conditions, it should be taken into account that due to the isometries of the identified faces, $d(A_0, A_1) = d(A_0, A_2)$ and $d(A_0, A_3) = d(A_1, A_3)$ hold. Based on of (\ref{distance}) (for the hyperbolic case), using the elements of matrix $\mathbb{A}$ we obtain
$$\frac{a_{01}}{\sqrt{a_{00} a_{11}}} = \frac{a_{02}}{\sqrt{a_{00} a_{22}}}, \ \ \ \frac{a_{03}}{\sqrt{a_{00} a_{33}}} = \frac{a_{13}}{\sqrt{a_{11} a_{33}}}$$
from which it follows, after squaring, that
$$\frac{a_{00} a_{11} - a_{01}^2}{a_{00} a_{11}} = \frac{a_{00} a_{22} - a_{02}^2}{a_{00} a_{22}}, \ \ \ \frac{a_{00} a_{33} - a_{03}^2}{a_{00} a_{33}} = \frac{a_{11} a_{33} - a_{13}^2}{a_{11} a_{33}}.$$

Using Lemma \ref{l:1} we get
$$a_{22} \left( b_{22} b_{33} - b_{23}^2 \right) = a_{11} \left( b_{11} b_{33} - b_{13}^2 \right), \ \ \ a_{11} \left( b_{11} b_{22} - b_{12}^2 \right) = a_{00} \left( b_{00} b_{22} - b_{02}^2 \right),$$
where
\begin{eqnarray*} 
a_{00} = a_{22} & = & \frac{B_{00}}{\Delta } = \frac{1}{\Delta } \left( 1 - \cos^2 \alpha_1 - \cos^2 \alpha_2 - \cos^2 \beta_1 - 2 \cos \alpha_1 \cos \alpha_2 \cos \beta_1  \right) > 0, \\
a_{11} = a_{33} & = & \frac{B_{11}}{\Delta } = \frac{1}{\Delta } \left( 1 - \cos^2 \alpha_1 - \cos^2 \beta_1 - \cos^2 \beta_2 - 2 \cos \alpha_1 \cos \beta_1 \cos \beta_2 \right) > 0.
\end{eqnarray*} 

Therefore, the edge conditions are 
\begin{equation} \label{edge}
f_1 \left( \alpha_1, \beta_1 \right) = 0, \ \ \ f_2 \left( \alpha_1, \beta_1 \right) = 0
\end{equation} 
where
\begin{eqnarray*} 
f_1 & = & \left( 1 - \cos^2 \alpha_1 - \cos^2 \alpha_2 - \cos^2 \beta_1 - 2 \cos \alpha_1 \cos \alpha_2 \cos \beta_1  \right) \sin^2 \beta_1  \\
    & - & \left( 1 - \cos^2 \alpha_1 - \cos^2 \beta_1 - \cos^2 \beta_2 - 2 \cos \alpha_1 \cos \beta_1 \cos \beta_2 \right) \sin^2 \alpha_2 \\
%\end{eqnarray*} 
%  
%\begin{eqnarray*}     
f_2 & = & \left( 1 - \cos^2 \alpha_1 - \cos^2 \beta_1 - \cos^2 \beta_2 - 2 \cos \alpha_1 \cos \beta_1 \cos \beta_2 \right) \sin^2 \alpha_1 \\
    & - &\left( 1 - \cos^2 \alpha_1 - \cos^2 \alpha_2 - \cos^2 \beta_1 - 2 \cos \alpha_1 \cos \alpha_2 \cos \beta_1  \right) \sin^2 \beta_2.
\end{eqnarray*} 

We assume that (\ref{AB}) is satisfied. Moreover, the solution for (\ref{edge}) have to be unique. By (\ref{AB}) the values of the angles have to satisfy $\alpha_1 \in (0, \frac{\pi}{a})$, $\alpha_2 \in (0, \frac{2 \pi}{a})$, $\beta_1 \in (0, \frac{\pi}{b})$, $\beta_2 \in (0, \frac{2 \pi}{b})$, although 
a more precise domain will be given in Section \ref{sec:4}.

Because of the squaring used, we need to check if the signs of $a_{01} = - \frac{B_{01}}{\Delta } $, $a_{02} = \frac{B_{02}}{\Delta }$, $a_{03} = - \frac{B_{03}}{\Delta }$ and $a_{13} = \frac{B_{13}}{\Delta }$ are the same. But it is easy to verify that all these signs are negative, for angles from the mentioned intervals.

%%%%%%%%%%%%%%%%%%%%
\subsection{The main Theorem and the strategy of proof}
\label{sec:3-2}

Consideration of the edge conditions (\ref{AB}), (\ref{edge}) and therefore the result of the following Theorem give the answer to the question of when the simplex $T$ is hyperbolic with vertices out of absolute.

\begin{theorem} \label{th:1}  
When $b > a \geq 2$, the simplex $T$ is hyperbolic with vertices out of absolute iff
\begin{equation} \label{main}
\left( 1 + \cos \frac{\pi}{a} \right) \sin \frac{2 \pi}{b} > \left( \cos \frac{\pi}{a} + \cos \frac{2 \pi}{b} \right) \sin \frac{\pi}{a} .
\end{equation} 
\end{theorem}

Note that the condition  (\ref{main}) is equivalent to condition $f_2 \left( \frac{\pi}{a}, 0 \right) > 0$, applying (\ref{AB}) to $\alpha_2$ and $\beta_2$. Also, note that in (\ref{main}) for fixed $a$, and increasing $b$, $\cos \frac{2 \pi}{b}$ increases and $\sin \frac{2 \pi}{b}$ decreases. Thus, the left side of the inequality decreases, and the right side increases. This means that for fixed $a$ there exists a corresponding $b_{max}$ such that the inequality (\ref{main}) is satisfied only for values $b \leq b_{max}$. We can see that this inequality is always satisfied for $b = 2a$, because then it becomes
$$\left( 1 + \cos \frac{\pi}{a} \right) \sin \frac{\pi}{a} > 2 \cos \frac{\pi}{a} \sin \frac{\pi}{a}. $$
Also, when $a \rightarrow \infty$, then $\cos \frac{\pi}{a} \rightarrow 1$, so for a sufficiently large $a$, $b_{max}$ becomes equal to $2a$. More precisely, by checking the individual values of $a$, the following values of $b_{max}$ were obtained:

%\begin{center}
\begin{tabular}{c|cccccccccccc} 
$a$ & 2 & 3 & 4 & 5 & 6 & 7 & 8 & 9 & 10 & 11 & 12 & \ldots  \\
\hline
$b_{max}$ & 7 & 8 & 9 & 11 & 12 & 14 & 16 & 18 & 20 & 22 & 24 & \ldots 
\end{tabular}

\vspace{5mm}

The case $a =2 $, $b = 3$ is examined by computer in \cite{S95}. It is shown that there are two solutions, but only one is proper. The values of the dihedral angles for the proper solution are included in the Table at the end of the paper. 

For other parameter cases, the contraction mapping theorem on the unique solution will be used in the proof of Theorem \ref{th:1}. So we shall first define

\begin{definition} \label{d:1} 
The function $\bar{g} = (g_1, g_2, \ldots, g_n)$ mapping a $n$-dimensional vector to $n$-dimensional vector is {\em mapping with contraction} if on closed and bounded set $D \subset \mathbb{R}^n$ it is true
\begin{enumerate}
	\item $\bar{x} = (x_1, x_2, \ldots, x_n) \in D \Rightarrow \bar{g} (\bar{x}) \in D$
	\item $\left\| \bar{g} (\bar{x}) - \bar{g} (\bar{y}) \right\| \leq L \left\| \bar{x} - \bar{y} \right\|$, for some $0 < L < 1$ and all $\bar{x}, \bar{y} \in D$.
\end{enumerate}
\end{definition}

The following Theorem then holds.

\begin{theorem} \label{th:2} 
If there exists closed and bounded set $D \subset \mathbb{R}^n$ such that $\bar{g}$ is mapping with contraction on it, then $\bar{x} = \bar{g} (\bar{x})$ has exactly one solution in $D$.
\end{theorem} 

It will be used the norm 
$\left\| \bar{x} \right\| = \max  \left| x_i \right|$, $1 \leq i \leq n$, 
and also the next Theorem. 

\begin{theorem} \label{th:3}
If the partial derivatives $\frac{\partial g_i}{\partial x_j}$ ($1 \leq i, j \leq n$) are continuous functions in closed and bounded set $D$, then for matrix $G = \left[ g_{ij} \right]_{n \times n}$ where $g_{ij}$ is supremum of the partial derivatives when $\bar{x} \in D$,
%$$g_{ij} = \sup_{\bar{x} \in D} \left| \frac{\partial g_i (\bar{x})}{\partial x_j} \right| $$
is true $\left\| \bar{g} (\bar{x}) - \bar{g} (\bar{y}) \right\| \leq \left\| G \right\| \left\| \bar{x} - \bar{y} \right\|$.
\end{theorem} 

%%%%%%%%%%%%%%%%%%%%%%%%%%%%
\section{The properties of functions $f_1$, $f_2$ and proof of the main Theorem %\ref{th:1}
}
\label{sec:4}

\subsection{First observations}

Here the first observations necessary for the proof will be given.

When $a = 2$, $b \geq 4$ it is fulfilled
\begin{equation} \label{5}
 \begin{array}{lcl}
f_1 \left( 0, \beta_1 \right) & = & -\left( 1 - \cos \beta_1 \right)^2 \sin ^2 \beta_1  \leq 0, \\
f_1 \left( \frac{\pi}{2}, \beta_1 \right) & = & - \cos ^2 \beta_1 \sin ^2 \beta_1 \leq 0, \\
f_1 \left( \alpha_1, 0 \right) & = & \left( \cos \alpha_1 + \cos \beta_2 \right)^2 \sin ^2 \alpha_2 \geq 0, \\
f_1 \left( \frac{\pi}{3}, \beta_1 \right) & = & - \frac{1}{16} + \frac{1}{4} \cos ^2 \beta_1 + \frac{3}{4} \cos ^2 \beta_2 + \frac{3}{4} \cos \beta_1 \cos \beta_2 - \\
  & - & \frac{1}{4} \sin ^2 2 \beta_1 - \frac{1}{2} \sin 2 \beta_1 \sin \beta_1 > 0.
 \end{array}
\end{equation}
In the last relation we estimate: if $\beta_2 \leq \frac{\pi}{3}$, then $\cos \beta_1 \geq \frac{\sqrt{2}}{2}$, $\cos \beta_2 \geq \frac{1}{2}$, $\sin \beta_1 \leq \frac{\sqrt{2}}{2}$, $\sin 2 \beta_1 \leq 1$ and so $f_1 \left( \frac{\pi}{3}, \beta_1 \right) \geq \frac{\sqrt{2}}{16} > 0$. When $b = 4, 5$ if $\beta_2 > \frac{\pi}{3}$, then $\beta_1 < \frac{\pi}{12}$, so again $f_1 \left( \frac{\pi}{3}, \beta_1 \right) > 0$. 

If $b > a \geq 3$ then 
\begin{equation} \label{6}
 \begin{array}{lcl}
f_1 \left( 0, \beta_1 \right) & = & -\left( \cos \frac{2 \pi}{a} + \cos \beta_1 \right)^2 \sin ^2 \beta_1 + \left( \cos \beta_1 + \cos \beta_2 \right)^2 \sin ^2 \frac{2 \pi}{a}
> 0, \\
f_1 \left( \frac{\pi}{a}, \beta_1 \right) & = & -\left( \cos \frac{\pi}{a} + \cos \beta_1 \right)^2 \sin ^2 \beta_1 \leq 0, \\
f_1 \left( \alpha_1, 0 \right) & = & \left( \cos \alpha_1 + \cos \frac{2 \pi}{b} \right)^2 \sin ^2 \alpha_2 \geq 0. 
   \end{array}
\end{equation}
The first relation is true because $\beta_1 \leq \frac{\pi}{b} < \frac{2 \pi}{a}$ and $\beta_2 \leq \frac{2 \pi}{b} < \frac{2 \pi}{a}$.

Also for $a \geq 2$, $b \geq 4$ it is true 
\begin{equation} \label{7}
 \begin{array}{lcl}
f_2 \left( 0, \beta_1 \right) & = & \left( \cos \frac{2 \pi}{a} + \cos \beta_1 \right)^2 \sin ^2 \beta_2 \geq 0, \\
f_2 \left( \alpha_1, \frac{\pi}{b} \right) & = & - \left( \cos \alpha_1 + \cos \frac{\pi}{b} \right)^2 \sin ^2 \alpha_1 \leq 0.
 \end{array}
\end{equation}

We shall also use partial derivatives $d_1$ of $f_1$ over $\alpha_1$ and $d_2$ of $f_2$ over $\alpha_2$, but taking into account (\ref{AB}). So, 
\begin{equation} \label{8}
 \begin{array}{lcl}
 d_1 & = & \frac{\partial f_1}{\partial \alpha_1} - 2 \frac{\partial f_1}{\partial \alpha_2} = \left[ \left( 2 \sin \alpha_1 \cos \alpha_1 + 2 \sin \alpha_1 \cos \alpha_2 \cos \beta_1 \right) \sin^2 \beta_1 \right. - \\
     & & \left.- \left( 2 \sin \alpha_1 \cos \alpha_1 + 2 \sin \alpha_1 \cos \beta_1 \cos \beta_2 \right) \sin^2 \alpha_2 \right] - 2 \left[ \left( 2 \sin \alpha_2 \cos \alpha_2 +    \right.\right.   \\
   & & + \left. 2 \sin \alpha_2 \cos \alpha_1 \cos \beta_1 \right) \sin^2 \beta_1 - 2 \sin \alpha_2 \cos \alpha_2 \left( 1 - \cos^2 \alpha_1 - \cos^2 \beta_1 - \right.\\
   & & \left.\left. - \cos^2 \beta_2 - 2 \cos \alpha_1 \cos \beta_1 \cos \beta_2 \right) \right] \\
 d_2 & = & \frac{\partial f_2}{\partial \beta_1} - 2 \frac{\partial f_2}{\partial \beta_2} = \left[ \left( 2 \sin \beta_1 \cos \beta_1 + 2 \sin \beta_1 \cos \alpha_1 \cos \beta_2 \right) \sin^2 \alpha_1 \right. - \\
     & & \left.- \left( 2 \sin \beta_1 \cos \beta_1 + 2 \sin \beta_1 \cos \alpha_1 \cos \alpha_2 \right) \sin^2 \beta_2 \right] - 2 \left[ \left( 2 \sin \beta_2 \cos \beta_2 +    \right.\right.   \\
   & & + \left. 2 \sin \beta_2 \cos \alpha_1 \cos \beta_1 \right) \sin^2 \alpha_1 - 2 \sin \beta_2 \cos \beta_2 \left( 1 - \cos^2 \alpha_1 - \cos^2 \alpha_2 - \right.\\
   & & \left.\left. - \cos^2 \beta_1 - 2 \cos \alpha_1 \cos \alpha_2 \cos \beta_1 \right) \right] .
 \end{array}
\end{equation}

\subsection{The domain $D$ of considerations}

As domain mentioned in Subsection \ref{sec:3-1} is not suitable for investigation for cases $a = 2$, $a = 3$, we shall make restriction to domain

\begin{equation} \label{9} D = \left\{
\begin{array}{l}
\left\{ \left( \alpha_1, \beta_1 \right): \alpha_1 \in \left[ \frac{\pi}{3}, \frac{\pi}{2} \right], \beta_1 \in \left[ 0, \frac{\pi}{b} \right]  \right\}, a = 2, b \geq 4, \\
\left\{ \left( \alpha_1, \beta_1 \right): \alpha_1 \in \left[ \frac{\pi}{12}, \frac{\pi}{3} \right], \beta_1 \in \left[ 0, \frac{\pi}{b} \right]  \right\}, b > a = 3,  \\
\left\{ \left( \alpha_1, \beta_1 \right): \alpha_1 \in \left[ 0, \frac{\pi}{a} \right], \beta_1 \in \left[ 0, \frac{\pi}{b} \right]  \right\}, b > a \geq 4.
\end{array}
\right.
\end{equation} 

That is the reason to prove the next

\begin{lemma} \label{l:2} The system (\ref{AB}), (\ref{edge})

\begin{enumerate}
	\item hasn't proper solution if $\alpha_1 \in \left[ 0, \frac{\pi}{3} \right]$ for $a = 2$, $b \geq 4$. 
	\item hasn't solution if $\alpha_1 \in \left[ 0, \frac{\pi}{12} \right]$ for $a = 3$, $b \geq 4$.
\end{enumerate}
\end{lemma}

\begin{proof}

{\bf 1.} If the solution exists for $a = 2$, $b \geq 4$ and $\alpha_1 \in \left[ 0, \frac{\pi}{3} \right]$, then $\alpha_2 \in \left[ \frac{\pi}{3}, \pi \right]$. Let us proof that $2 \alpha_1 < \beta_1 $.

From (\ref{edge}) follows $\frac{\sin \alpha_2}{\sin \alpha_1} = \frac{\sin \beta_1}{\sin \beta_2}$. If valid $a = 2$ then
$$\cos \alpha_1 = \frac{\sin \beta_1}{2 \sin \beta_2},$$ 
because $\sin \alpha_2 = \sin 2 \alpha_1$, $\cos \alpha_2 = - \cos 2 \alpha_1$.
That means, if $\alpha_1 < \frac{\pi}{3}$, then $\sin \beta_1 \geq \sin \beta_2$. So, we can assume that $\beta_2 < \frac{2 \pi}{3 b} < \beta_1 < \frac{\pi}{4}$.	
	
We know that $f_1 = B_{00} \sin^2 \beta_1 - B_{11} \sin^2 \alpha_2$ and as it was mentioned in Subsection \ref{sec:3-1}, $B_{00}, B_{11} < 0$. If $2 \alpha_1 = \beta_1$ then it holds $B_{00} > B_{11}$ because $\cos \beta_2 > \left| \cos \alpha_2 \right|$, and $\sin^2 \beta_1 = \sin^2 \alpha_2$. Consequence is that for $2 \alpha_1 = \beta_1$, $f_1 > 0$. 

On the other hand, by (\ref{5}) $f_1 \left( 0, \beta_1 \right) < 0$ ($\beta_1 \neq 0$). So, for solution of the system (\ref{AB}), it have to be valid $2 \alpha_1 < \beta_1 $.

Since $\Delta, B_{00} < 0$, if the solution is proper, the simplex $T$ has to be hyperbolic with vertices out of absolute. Then after cutting the simplex $T$ by the polar plane of the vertex $A_0$ the vertex figure is a hyperbolic triangle with angles equal to the dihedral angles of the edges $A_0 A_1$, $A_0 A_2$, $A_0 A_3$, i.e. the angles $\alpha_1$, $\alpha_2$, $\beta_1$. As $2 \alpha_1 + \alpha_2 = \pi$ and $2 \alpha_1 < \beta_1 $, their sum is greater then $3 \alpha_1 + \alpha_2 > \pi$, what is a contradiction. This means the system (\ref{AB}), (\ref{edge}) have the solution which is not proper.

\vspace{3mm}

{\bf 2.} When $b > a \geq 3$, let us consider $f_1$ written in the form
\begin{eqnarray*} f_1 \left( \alpha_1, \beta_1 \right) & = & \left( 1 - \cos^2 \alpha_1 - \cos^2 \beta_1 \right) \left( \sin^2 \beta_1 - \sin^2 \alpha_2 \right) +\\
   & & + \left( \cos^2 \beta_2 + 2 \cos \alpha_1 \cos \beta_1 \cos \beta_2 \right) \sin^2 \alpha_2 - \\
   & & - \left( \cos^2 \alpha_2 + 2 \cos \alpha_1 \cos \alpha_2 \cos \beta_1  \right) \sin^2 \beta_1.
\end{eqnarray*}

If $\alpha_1 < \frac{\pi}{12}$ then $\alpha_2 \in \left[ \frac{\pi}{2} ,\frac{2 \pi}{3} \right]$. It is also true $\beta_1 \leq \frac{\pi}{4}$ and $\beta_2 \leq \frac{\pi}{2}$. So, $1 - \cos^2 \alpha_1 - \cos^2 \beta_1 < 0$, $\sin^2 \beta_1 < \sin^2 \alpha_2 $, $\cos \alpha_2 < 0$, $\cos \alpha_2 + 2 \cos \alpha_1 \cos \beta_1 > \frac{1}{2} > 0$. This means $f_1 \left( \alpha_1, \beta_1 \right) > 0$ for all $\alpha_1 < \frac{\pi}{12}$ and therefore the system (\ref{AB}), (\ref{edge}) hasn't any solution.
\end{proof}

%%%%%%%%%%%%%%%%%%%%%%%%

\subsection{The properties of $d_1$ and $d_2$}

%%%%%%%%%%%%%%%%%%%%%

We shall first consider some values of derivatives $d_1$ and $d_2$, where we think about the function $d_1$ as a three parameters function with parameters $\alpha_1$, $\alpha_2$, $\beta_1$ (it still holds $\beta_2 = \frac{2 \pi}{b} - 2 \beta_1$, by (\ref{AB})) and about $d_2$ as a function with parameters $\alpha_1$, $\beta_1$, $\beta_2$ ($\alpha_2 = \frac{2 \pi}{a} - 2 \alpha_1$ by (\ref{AB})).

\vspace{3mm}

%{\bf 1.} 
For $d_1$ it is true
$$d_1 \left( \alpha_1, 0, \beta_1 \right) = 2 \left( \sin \alpha_1 \cos \alpha_1 + \sin \alpha_1 \cos \beta_1 \right) \sin^2 \beta_1 \geq 0.$$

On the other hand for $a = 2$, $b \geq 4$ it holds
%%%%%%%   a = 2
$$ \begin{array}{c}
d_1 \left( \alpha_1, \frac{\pi}{3}, \beta_1 \right) = \left( 2 \sin \alpha_1 \cos \alpha_1 + \sin \alpha_1 \cos \beta_1 - \sqrt{3} - 2 \sqrt{3} \cos \alpha_1 \cos \beta_1 \right) \sin^2 \beta_1 - \\
   -  \frac{3}{4} \left( 2 \sin \alpha_1 \cos \alpha_1 + 2 \sin \alpha_1 \cos \beta_1 \cos \beta_2 \right) + \ \ \ \ \ \ \ \   \\
\ \ \ \ \ \ \ \ \ \ \ \ \ \ \ \ \ \ \ \  
   +  \sqrt{3} \left( 1 - \cos^2 \alpha_1 - \cos^2 \beta_1 - \cos^2 \beta_2 - 2 \cos \alpha_1 \cos \beta_1 \cos \beta_2  \right) < 0,
\end{array} $$
because $\sin^2 \beta_1 \leq \frac{3}{4}$, $\sin \alpha_1 \cos \alpha_1 < 1 < \sqrt{3}$. Also from $\beta_2 = \frac{2 \pi}{b} - 2 \beta_1 \leq \frac{\pi}{2} - 2 \beta_1$ follows $\cos \beta_2 \geq \sin 2 \beta_1 \geq \sin \beta_1$ and therefore holds $\cos^2 \beta_1 + \cos^2 \beta_2 \geq \cos^2 \beta_1 + \sin^2 \beta_1 = 1$.

\vspace{3mm}

%%%%%%%%%%%%%%%%%%%%%%%%%%%5

If $b > a = 3$ is true, then $\sin \alpha_1 < 1 < \sqrt{2} \leq 2 \cos \beta_1$, so
%%%%%%%   a = 3 xxxx
$$ \begin{array}{c}
d_1 \left( \alpha_1, \frac{\pi}{2}, \beta_1 \right) = \left( 2 \sin \alpha_1 \cos \alpha_1 - 4 \cos \alpha_1 \cos \beta_1 \right) \sin^2 \beta_1 -  \left( 2 \sin \alpha_1 \cos \alpha_1 + \right. \\
\left. + 2 \sin \alpha_1 \cos \beta_1 \cos \beta_2 \right) < 0.
\end{array} $$ 

\vspace{3mm}

%%%%%%%%%%%%%%%% a=4!!!!!!!!!!!
In the case $b > a \geq 4$ it is valid
$$ \begin{array}{c}
d_1 \left( \alpha_1, \frac{2 \pi}{a}, \beta_1 \right) = \left( 2 \sin \alpha_1 \cos \alpha_1 + 2 \sin \alpha_1 \cos \frac{2 \pi}{a} \cos \beta_1 - 4 \sin \frac{2 \pi}{a} \cos \frac{2 \pi}{a} - \right. \\
\left.  - 4 \sin \frac{2 \pi}{a} \cos \alpha_1 \cos \beta_1  \right) \sin^2 \beta_1 - \left( 2 \sin \alpha_1 \cos \alpha_1 + 2 \sin \alpha_1 \cos \beta_1 \cos \beta_2 \right) \sin^2 \frac{2 \pi}{a} \\
+ 4 \sin \frac{2 \pi}{a} \cos \frac{2 \pi}{a} \left( 1 - \cos^2 \alpha_1 - \cos^2 \beta_1 - \cos^2 \beta_2 - 2 \cos \alpha_1 \cos \beta_1 \cos \beta_2 \right) < 0,
\end{array} $$
because
$$1 - \cos^2 \alpha_1 - \cos^2 \beta_1 - \cos^2 \beta_2 - 2 \cos \alpha_1 \cos \beta_1 \cos \beta_2 < 0$$
and from $\cos \beta_1 > \frac{1}{2}$, $\sin \frac{2 \pi}{a} > \sin \frac{\pi}{a} \geq \sin \alpha_1$ also holds
$$\sin \alpha_1 \cos \alpha_1 < 2 \sin \frac{2 \pi}{a} \cos \alpha_1 \cos \beta_1, \ \ \
\sin \alpha_1 \cos \frac{2 \pi}{a} \cos \beta_1 < \sin \frac{2 \pi}{a} \cos \frac{2 \pi}{a}.$$

%%%%%%%%%%%%%%%%%%%%%%%%%%%%%%
%%%%%%%%%%%%%%%%%%%%%%%%%%%%%

%{\bf 2.} 
For $d_2$ similarly as before holds
$d_2 \left( \alpha_1, \beta_1, 0 \right) \geq 0$, and $d_2 \left( \alpha_1, \beta_1, \frac{2 \pi}{b} \right) < 0$.

%%%%%%%%%%%%%%%%%%%%%%%%%%%%%%%%%
%%%%%%%%%%%%%%%%%%%%%%%%%%%%%%%5

\vspace{5mm}

For derivatives $d_1$ and $d_2$ the next Lemma is valid.

\begin{lemma} \label{l:3}
When (\ref{AB}) applied (complete, for both $\alpha_2$, $\beta_2$), in domain $D$ holds:
\begin{enumerate}
	\item For a fixed $\bar{\beta_1}$, there is a unique value $\alpha_1 = \alpha_d$ such that $d_1 \left( \alpha_d, \bar{\beta_1} \right) = 0$.
	\item For a fixed $\bar{\alpha_1}$, there is a unique value $\beta_1 = \beta_d$ such that $d_2 \left( \bar{\alpha_1}, \beta_d \right) = 0$.
\end{enumerate}
\end{lemma}

\begin{proof} 
{\bf 1.} For the system $d_1 \left( \alpha_1, \alpha_2, \beta_1 \right) = 0$, $2 \alpha_1 + \alpha_2 = \frac{2 \pi}{a}$ we shall form the equivalent one, using fixed value of $\beta_1 = \bar{\beta_1}$ (and the corresponding fixed value of $\beta_2$), as well as a conveniently chosen constant $k_1 > 0$, which ensures that within the Definition \ref{d:1} is fulfilled $0 < L < 1$:
\begin{equation} \label{10}
\begin{array} {lcl}
   \alpha_1 & = & g_1^1 \left( \alpha_1, \alpha_2 \right) \\
   \alpha_2 & = & g_2^1 \left( \alpha_1, \alpha_2 \right)
\end{array}
\end{equation}
where 
$$ \begin{array} {ccl} %\begin{eqnarray*}
   g_1^1 \left( \alpha_1, \alpha_2 \right) & = & \frac{\pi}{a} - \frac{\alpha_2}{2} \\
   g_2^1 \left( \alpha_1, \alpha_2 \right) & = & \frac{3 \alpha_2}{4} - \frac{\pi}{2a} - \frac{\alpha_1}{2} + \frac{d_1 \left( \alpha_1, \alpha_2, \bar{\beta_1} \right)}{k_1}.
%\end{eqnarray*}
\end{array} $$
How $d_1$ is a continuous and bounded function on the domain
%%%%%%%%%%%%%%%
\begin{equation*}  D_1 = \left\{
\begin{array}{l}
\left\{ \left( \alpha_1, \alpha_2 \right): \alpha_1 \in \left[ \frac{\pi}{3}, \frac{\pi}{2} \right], \alpha_2 \in \left[ 0, \frac{\pi}{3} \right]  \right\}, a = 2, b \geq 4, \\
\left\{ \left( \alpha_1, \alpha_2 \right): \alpha_1 \in \left[ \frac{\pi}{12}, \frac{\pi}{3} \right], \alpha_2 \in \left[ 0, \frac{\pi}{2} \right]  \right\}, b > a = 3,  \\
\left\{ \left( \alpha_1, \alpha_2 \right): \alpha_1 \in \left[ 0, \frac{\pi}{a} \right], \alpha_2 \in \left[ 0, \frac{2 \pi}{a} \right]  \right\}, b > a \geq 4,
\end{array}
\right.
\end{equation*} 
%%%%%%%%%%%%
the function $\bar{g}^1 = \left( g_1^1, g_2^1 \right)$ is mapping with contraction. Therefore the system (\ref{10}) has a unique solution $\alpha_1 = \alpha_d$, $\alpha_2 = \frac{2 \pi}{a} - 2 \alpha_d$ which is also giving a solution $\alpha_1 = \alpha_d$ of the equation $d_1 \left( \alpha_1, \bar{\beta_1} \right) = 0$.

\vspace{5mm}

{\bf 2.} Similarly, we shall start with a system $d_2 \left( \alpha_1, \beta_1, \beta_2 \right) = 0$, $2 \beta_1 + \beta_2 = \frac{2 \pi}{b}$, and using fixed value of $\alpha_1 = \bar{\alpha_1}$ and a conveniently chosen constant $k_2 > 0$ form equivalent system
\begin{equation*} 
\begin{array} {lclcl}
   \beta_1 & = & g_1^2 \left( \beta_1, \beta_2 \right) & = & \frac{\pi}{b} - \frac{\beta_2}{2}\\
   \beta_2 & = & g_2^2 \left( \beta_1, \beta_2 \right) & = & \frac{3 \beta_2}{4} - \frac{\pi}{2b} - \frac{\beta_1}{2} + \frac{d_2 \left( \bar{\alpha_1}, \beta_1, \beta_2 \right)}{k_1}
\end{array}
\end{equation*}
which similarly have a unique solution on the domain 
$$D_2 = \left\{ \left( \beta_1, \beta_2 \right): \beta_1 \in \left[ 0, \frac{\pi}{b} \right], \beta_2 \in \left[ 0, \frac{2 \pi}{b} \right]  \right\},$$ 
leading to a unique solution $\beta_1 = \beta_d$ of  $d_2 \left( \bar{\alpha_1}, \beta_1 \right) = 0$.
\end{proof}

%%%%%%%%%%%%%%%%%%%%%%5555
%%%%%%%%%%%%%%%%

\subsection{The proof of the Theorem \ref{th:1}}

Using properties of $d_1$ and $d_2$ given in the previous subsection it is easy to proof the following 

\begin{lemma} \label{l:4}
On the domain $D$ it is fulfilled
\begin{enumerate}
  \item For the fixed $\beta_1 = \bar{\beta_1}$ the equation $f_1 \left( \alpha_1, \bar{\beta_1} \right) = 0$ has a unique solution $\alpha_1 = \alpha_f$ such that $0 < \alpha_f < \alpha_d < \frac{\pi}{a}$ for $\bar{\beta_1} \neq 0$ and $\alpha_f = \alpha_d = \frac{\pi}{a}$ for $\bar{\beta_1} = 0$.
  \item If $f_2 \left( \bar{\alpha_1}, 0 \right) \geq 0$ is true for the fixed $\alpha_1 = \bar{\alpha_1}$ then the equation $f_2 \left( \bar{\alpha_1}, \beta_1 \right) = 0$ has a unique solution $\beta_1 = \beta_f$ such that $0 \leq \beta_f < \beta_d < \frac{\pi}{b}$ for $\bar{\alpha_1} \neq 0$ and $\beta_f = \beta_d = \frac{\pi}{b}$ for $\bar{\alpha_1} = 0$. If $f_2 \left( \bar{\alpha_1}, 0 \right) < 0$ is satisfied for the fixed $\alpha_1 = \bar{\alpha_1}$ then the equation $f_2 \left( \bar{\alpha_1}, \beta_1 \right) = 0$ has no solutions.  
\end{enumerate}
\end{lemma}

\begin{proof} 
{\bf 1.} By Lemma \ref{l:3} for the fixed $\bar{\beta_1} \neq 0$ the function $f_1$ decreases for $\alpha_1 \in \left( 0, \alpha_d \right)$ and increases for $\alpha_1 \in \left( \alpha_d, \frac{\pi}{a} \right)$. For $\bar{\beta_1} = 0$ the function $f_1$ decreases. It follows that $f_1 \left( \alpha_1, \bar{\beta_1} \right) = 0$ has a unique solution.

\vspace{3mm} 

{\bf 2.} The first part of the statement is proved similarly to the first part of the Lemma. If $f_2 \left( \bar{\alpha_1}, 0 \right) < 0$ holds, it is easy to see that $f_2 \left( \bar{\alpha_1}, \beta_1 \right) < 0$ for each $\beta_1 \in \left[ 0, \frac{\pi}{b} \right]$.

\end{proof} 

%%%%%%%%%%%%%%%%%%%%%%%%%%%%%   !!!!!!!!!

In the following Lemmas \ref{l:5} - \ref{l:8}, four cases are given which complete the proof of Theorem \ref{th:1}. We assume that (\ref{AB}) holds in all of these Lemmas.

%%%%%%%%%%%%%%%%%%%%%%%%%%

\begin{lemma} \label{l:5}
If $f_2 \left( \frac{\pi}{a}, 0 \right) < 0$ then the system (\ref{edge}) has no solutions.
\end{lemma}

%%%%%%%%%%%%%%%%%%
\begin{proof} 
When $a = 2$, $b \geq 4$, then $f_2 \left( \frac{\pi}{a}, 0 \right) < 0$ is true only if $b \geq 9$. Then
\begin{eqnarray*}
f_2 \left( \alpha_1, 0 \right) & = & - \left( \cos \alpha_1 + \cos \beta_2 \right)^2 \sin^2 \alpha_1 + \left( \cos \alpha_1 + \cos \alpha_2 \right)^2 \sin^2 \beta_2 < \\
  	& < & \left( \cos \alpha_1 + \frac{\sqrt{2}}{2} \right)^2 \sin^2 \alpha_1 + \left( \cos \alpha_1 + \cos \alpha_2 \right)^2 \frac{1}{2} = F \left( \cos \alpha_1 \right),
\end{eqnarray*}
with $\alpha_1 \in \left[ \frac{\pi}{3}, \frac{\pi}{2} \right]$. After replacement $x = \cos \alpha_1$, $x \in \left[ 0, \frac{1}{2} \right]$ ,
$$F(x) = - \left( x + \frac{\sqrt{2}}{2} \right)^2 \left( 1 - x^2 \right) + \frac{1}{2} \left( x - 2 x^2 + 1 \right)^2. $$

On the interval $\left[ 0, \frac{1}{2} \right]$, $F(x)$ decreases because $F' = 12 x^3 + 3\left( \sqrt{2} - 2 \right) x^2 -4x + 1 - \sqrt{2}$; $F'\left( - \frac{1}{4} \right) > 0$, $F'\left( 0 \right) < 0$, $F'\left( \frac{1}{2} \right) < 0$, $F'\left( 1 \right) > 0$. As $F\left( 0 \right) = 0$, it follows $f_2 \left( \alpha_1, 0 \right) < 0$ for each $\alpha_1 \in \left[ \frac{\pi}{3}, \frac{\pi}{2} \right]$. Therefore, by Lemma \ref{l:4} the system (\ref{edge}) has no solutions.

%%%%%%%%%%
\vspace{5mm}

If $b > a \geq 3$, by Lemma \ref{l:4} follows $f_2 \left( \frac{\pi}{a}, \beta_1 \right) < 0$ for each $\beta_1 \in \left[ 0, \frac{\pi}{b} \right]$. Also, by (\ref{6}), (\ref{7}) it is $f_1 \left( 0, \beta_1 \right) > 0$, $f_1 \left( \frac{\pi}{a}, \beta_1 \right) \leq 0$, $f_2 \left( 0, \beta_1 \right) \geq 0$. 

We shall proof that on the diagonal $d: \beta_1 = - \frac{a \alpha_1}{b} + \frac{\pi}{b}$ (or $\alpha_1 = - \frac{b \beta_1}{a} + \frac{\pi}{a}$) of the domain 
$$D^+ = \left\{ \left( \alpha_1, \beta_1 \right): \alpha_1 \in \left[ 0, \frac{\pi}{a} \right] , \beta_1 \in \left[ 0, \frac{\pi}{b} \right] \right\} $$
is true $f_1 \left( \alpha_1, \beta_1 \right) > 0$ for $\left( \alpha_1, \beta_1 \right) \neq \left( \frac{\pi}{a}, 0 \right)$ and $f_1 \left( \frac{\pi}{a}, 0 \right) = 0$. Also $f_2 \left( \alpha_1, \beta_1 \right) < 0$ for $\left( \alpha_1, \beta_1 \right) \neq \left( 0, \frac{\pi}{b} \right)$ and $f_2 \left( 0, \frac{\pi}{b} \right) = 0$. That would mean, by Lemma \ref{l:4} that system (\ref{edge}) has no solutions. More precisely, the guaranteed unique value $\alpha_f$ for a given $\beta_1$ satisfying $f_1 \left( \alpha_f, \beta_1 \right) = 0$ is greater then $- \frac{b \beta_1}{a} + \frac{\pi}{a}$ (i.e. above the diagonal, if we assume that $\alpha_1$ and $\beta_1$ are placed on the $x$ and $y$ axes, respectively). Similarly, if for some $\alpha_1$ is fulfilled $f_2 \left( \alpha_1, 0 \right) \geq 0$, then the guaranteed unique value $\beta_f$ giving $f_2 \left( \alpha_1, \beta_f \right) = 0$ is less than $- \frac{a \alpha_1}{b} + \frac{\pi}{b}$ (i.e. it is below the diagonal). So, $f_1$ and $f_2$ cannot be zero at the same time.

Let us consider the systems $\beta_1 = - \frac{a \alpha_1}{b} + \frac{\pi}{b}$, $f_1 \left( \alpha_1, \beta_1 \right) = 0$ and $\beta_1 = - \frac{a \alpha_1}{b} + \frac{\pi}{b}$, $f_2 \left( \alpha_1, \beta_1 \right) = 0$. We shall transform them to equivalent systems $\alpha_1 = g_1^1$, $\beta_1 = g_2^1$ and $\alpha_1 = g_1^2$, $\beta_1 = g_2^2$ where
\begin{equation*} 
\begin{array} {lcl}
g_1^1 = (1 - t) \alpha_1 + t (- \frac{b \beta_1}{a} + \frac{\pi}{a}) + \frac{f_1}{k_1} & \ \ \ &
g_1^2 = (1 - t) \alpha_1 + t (- \frac{b \beta_1}{a} + \frac{\pi}{a}) + \frac{f_2}{k_2} \\
g_2^1 = - \frac{a \alpha_1}{b} + \frac{\pi}{b} & \ \ \ &
g_2^2 = - \frac{a \alpha_1}{b} + \frac{\pi}{b}
\end{array} 
\end{equation*}
If we take for example $t = \frac{a}{2b}$, and choose constants $k_1$, $k_2$, so that the partial derivatives are less then 1, and to be satisfied $g_1^1, g_2^1 \in \left[ 0, \frac{\pi}{a} \right]$, then both functions $\bar{g}_1 = (g_1^1, g_2^1)$, $\bar{g}_2 = (g_1^2, g_2^2)$ are mapping with contraction on domain $D^+$. That means both of starting systems have unique solution $f_1 \left( \frac{\pi}{a}, 0 \right) = 0$ and $f_2 \left( 0, \frac{\pi}{b} \right) = 0$ respectively.
\end{proof} 

%%%%%%%%%%%%%%%%%

\begin{lemma} \label{l:6}
If $f_2 \left( \frac{\pi}{a}, 0 \right) = 0$ then the system (\ref{edge}) has no proper solutions.
\end{lemma}

%%%%%%%%%%%%%%%%%%
\begin{proof} 
In this case, $\left( \alpha_1, \beta_1 \right) = \left( \frac{\pi}{a}, 0 \right)$ is solution of the system (\ref{edge}), but it is not proper, because the dihedral angles of simplex $T$ cannot be equal to zero. Let us proof that this solution is unique.

When $a = 2$, the assumption of Lemma is fulfilled only for $b = 8$.
Then similarly as in the poof of previous Lemma it can be shown that for $\alpha_1 \in \left[ \frac{\pi}{3}, \frac{\pi}{2} \right]$ and $\left( \alpha_1, \beta_1 \right) \neq \left( \frac{\pi}{2}, 0 \right)$ holds $f_2 \left( \alpha_1, \beta_1 \right) < 0$.

For $b > a \geq 3$ situation is also similar as in the poof of previous Lemma, but on the diagonal $d$ the function $f_2$ has two zero values, $f_2 \left( 0, \frac{\pi}{b} \right) = 0$ and $f_2 \left( \frac{\pi}{a}, 0 \right) = 0$. Therefore we shall notice that, because of continuity of the function $f_1$, there exist an arbitrarily small number $\lambda > 0$ such that for $\alpha_1 \in \left[ 0, \lambda \right]$, $f_1 \left( \alpha_1, \beta_1 \right) > 0$ is true.

Then, we can use instead of $d$ the straight line $d_{\lambda}$ passing through the points $\left( \lambda, \frac{\pi}{b} \right)$ and $\left( \frac{\pi}{a}, 0 \right)$, that is $d_{\lambda}: \beta_1 = \frac{ \left( \alpha_1 - \frac{\pi}{a} \right) a \pi }{b \left( \lambda a - \pi \right)}$. Afterwards, similarly as before we shall prove that on $d_{\lambda}$ is true $f_1 \left( \alpha_1, \beta_1 \right) > 0$ and $f_2 \left( \alpha_1, \beta_1 \right) < 0$ for $\left( \alpha_1, \beta_1 \right) \neq \left( \frac{\pi}{a}, 0 \right)$.

Then on the domain 
$$D_{\lambda} = \left\{ \left( \alpha_1, \beta_1 \right): \alpha_1 \in \left[ \lambda , \frac{\pi}{a} \right] , \beta_1 \in \left[ 0, \frac{\pi}{b} \right] \right\}$$
the functions $\bar{g}_1 = (g_1^1, g_2^1)$, $\bar{g}_2 = (g_1^2, g_2^2)$ where
%%%
$$ \begin{array}{ccl}
g_1^1 & = & (1 - t) \alpha_1 + t \left( \frac{\pi}{a} + \frac{b \left( \lambda a - \pi \right) \beta_1 }{a \pi} \right) + \frac{f_1}{k_1} \\
g_1^2 & = & (1 - t) \alpha_1 + t \left( \frac{\pi}{a} + \frac{b \left( \lambda a - \pi \right) \beta_1 }{a \pi} \right) + \frac{f_2}{k_2} \\
g_2^1 & = & g_2^2  \ = \ \frac{ \left( \alpha_1 - \frac{\pi}{a} \right) a \pi }{b \left( \lambda a - \pi \right)}
%, \ \alpha_1 \geq \lambda; \ \ \ g_2^2 = \frac{\pi}{b}, \alpha_1 < \lambda 
\end{array} $$
%%%
are mapping with contraction. As $\lambda$ is arbitrarily small number, we can take $t = \frac{a}{2b}$ and assume that $\left| \frac{a \pi}{b \left( \lambda a - \pi \right)} \right| < 1$. 
Therefore, similarly as in the poof of previous Lemma, $\left( \alpha_1, \beta_1 \right) = \left( \frac{\pi}{a}, 0 \right)$ is the unique solution of the system (\ref{edge}).
\end{proof}

%%%%%%%%%%%%%%%%%

\begin{lemma} \label{l:7}
If $f_2 \left( \alpha_1, 0 \right) > 0$ for each $\alpha_1 \in \left[ 0, \frac{\pi}{a} \right]$ then the system (\ref{edge}) has unique solution.
\end{lemma}

%%%%%%%%%%%%%%%%%%
\begin{proof} 
By (\ref{6}) for $b > a \geq 3$ holds $f_1 \left( 0, \beta_1 \right) > 0$ and by the assumption of this Lemma holds $f_2 \left( \alpha_1, 0 \right) > 0$. As $f_1$, $f_2$ are continuous functions, there exist such values $\lambda > 0$, $\mu > 0$ that for $\alpha_1 \in \left[ 0, \lambda \right]$ holds $f_1 \left( \alpha_1, \beta_1 \right) > 0$ and for $\beta_1 \in \left[ 0, \mu \right]$ holds $f_2 \left( \alpha_1, \beta_1 \right) > 0$. We can use $\lambda = \frac{\pi}{12}$ for $b > a = 3$ and $\lambda = \frac{\pi}{3}$ for $a = 2$, $b \geq 4$ in accordance with the Lemma \ref{l:2}.

Due to continuity of the functions $f_1$, $f_2$, $d_1$, $d_2$, by the Lemmas \ref{l:3}, \ref{l:4} there exists continuous functions $\alpha_1 = l_1 \left( \beta_1 \right)$ and $\beta_1 = l_2 \left( \alpha_1 \right)$ such that (using notations from the previous Lemmas) for fixed $\bar{\beta_1} \geq \mu$ holds $\alpha_f < l_1 \left( \bar{\beta_1} \right) < \alpha_d$ and for fixed $\bar{\alpha_1} \geq \lambda$ holds $\beta_f < l_2 \left( \bar{\alpha_1} \right) < \beta_d$. So, for the points on the line $\alpha_1 = l_1 \left( \beta_1 \right)$ is true $f_1 \left( \alpha_1, \beta_1 \right) < 0$ and $d_1 \left( \alpha_1, \beta_1 \right) < 0$, while for the points on the line $\beta_1 = l_2 \left( \alpha_1 \right)$ is true $f_2 \left( \alpha_1, \beta_1 \right) < 0$ and $d_2 \left( \alpha_1, \beta_1 \right) < 0$. Therefore, on the closed domain $D^-$ bounded by $\alpha_1 = \lambda$, $\alpha_1 = l_1 \left( \beta_1 \right)$, $\beta_1 = \mu$, $\beta_1 = l_2 \left( \alpha_1 \right)$ function $\bar{g} = (g_1, g_2)$ where $g_1 = \alpha_1 + \frac{f_1}{k_1}$, $g_2 = \beta_1 + \frac{f_2}{k_2}$ ($k_1$, $k_2$ are conveniently chosen constants) is mapping with contraction. Then the system $\alpha_1 = g_1 \left( \alpha_1, \beta_1 \right)$, $\beta_1 = g_2 \left( \alpha_1, \beta_1 \right)$ which is equivalent to the system (\ref{edge}) has the unique solution on $D^-$. Note, that boundaries of $D^-$ are so chosen that outside of it either there are no solutions of the system or they are no proper.
\end{proof}

%%%%%%%%%%%%%%%%%

\begin{lemma} \label{l:8}
If $f_2 \left( \frac{\pi}{a}, 0 \right) > 0$ and there exists $\hat{\alpha_1} \in \left[ 0, \frac{\pi}{a} \right)$ such that $f_2 \left( \hat{\alpha_1}, 0 \right) \leq 0$
then the system (\ref{edge}) has unique solution.
\end{lemma}

%%%%%%%%%%%%%%%%%%
\begin{proof} 
For parameters $a = 2$, $b \geq 4$ and $b > a = 3$, due to the established in Lemma \ref{l:2}, it is of interest to investigate only the cases of $\hat{\alpha_1} \in \left[ \frac{\pi}{3}, \frac{\pi}{2} \right)$ and $\hat{\alpha_1} \in \left[ \frac{\pi}{12}, \frac{\pi}{3} \right)$, respectively. Since $f_1 \left( \alpha_1, 0 \right) > 0$ 
for $\alpha_1 \neq \frac{\pi}{a}$ %(and $\alpha_1 \neq 0$ for $a = 2$, $b \geq 4$),
and $f_2 \left( \frac{\pi}{a}, 0 \right) > 0$, the system (\ref{edge}) has no solution for $\beta_1 = 0$. As the functions $f_1$ and $f_2$ are are continuous, there exists such value $\mu > 0$ such that the system (\ref{edge}) has no solution for $\beta_1 \in \left[ 0, \mu \right]$. We shall additionally require that $\mu$ satisfies $f_2 \left( \frac{\pi}{a}, \mu \right) > 0$. Let $\lambda \in \left( 0, \frac{\pi}{a} \right)$ be such that for $\alpha_1 \in \left( \lambda, \frac{\pi}{a} \right]$ is fulfilled $f_2 \left( \alpha_1, \mu \right) > 0$ and $f_2 \left( \lambda, \mu \right) = 0$. 

Similarly as in the proof of the previous Lemma, let the lines $\alpha_1 = l_1 \left( \beta_1 \right)$ and $\beta_1 = l_2 \left( \alpha_1 \right)$ be such that for the fixed $\bar{\beta_1} \geq \mu$ holds $\alpha_f < l_1 \left( \bar{\beta_1} \right) < \alpha_d$ and for fixed $\bar{\alpha_1} \geq \lambda$ holds $\beta_f < l_2 \left( \bar{\alpha_1} \right) < \beta_d$. Because of $f_2 \left( \lambda, \mu \right) = 0$ (that means $\beta_f = \mu$ for $\bar{\alpha_1} = \lambda$), it is possible to chose the line $\beta_1 = l_2 \left( \alpha_1 \right)$ such that $l_2 \left( \lambda \right) \leq \frac{- \lambda a}{b} + \frac{\pi}{b}$. Then, using previously proven in Lemma \ref{l:5} (on the diagonal $d$ holds $f_1 > 0$, for $\alpha_1 \neq \frac{\pi}{a}$) here we get $f_1 \left( \lambda, \beta_1 \right) > 0$, for $\beta_1 \in \left[ \mu, l_2 \left( \lambda \right) \right]$. Also, for points on the line $\alpha_1 = l_1 \left( \beta_1 \right)$ is true 
$f_1 \left( \alpha_1, \beta_1 \right) < 0$, $d_1 \left( \alpha_1, \beta_1 \right) < 0$, and for the points on the line $\beta_1 = l_2 \left( \alpha_1 \right)$ is true $f_2 \left( \alpha_1, \beta_1 \right) < 0$ and $d_2 \left( \alpha_1, \beta_1 \right) < 0$.

Therefore, on the closed domain $\bar{D}$ bounded by $\alpha_1 = \lambda$, $\alpha_1 = l_1 \left( \beta_1 \right)$, $\beta_1 = \mu$, $\beta_1 = l_2 \left( \alpha_1 \right)$ function $\bar{g} = (g_1, g_2)$ where $g_1 = \alpha_1 + \frac{f_1}{k_1}$, $g_2 = \beta_1 + \frac{f_2}{k_2}$ ($k_1$, $k_2$ are conveniently chosen constants) is mapping with contraction. Then the system $\alpha_1 = g_1 \left( \alpha_1, \beta_1 \right)$, $\beta_1 = g_2 \left( \alpha_1, \beta_1 \right)$ which is equivalent to the system (\ref{edge}) has the unique solution on $\bar{D}$. 

It only remains to prove that system (\ref{edge}) has no solutions on the domain
$$D^* = \left\{ \left( \alpha_1, \beta_1 \right): \alpha_1 \in \left[ 0, \lambda \right] , \beta_1 \in \left[ \mu, \frac{\pi}{b} \right] \right\} .$$

We shall first proof that on the part of diagonal $d$ ($d: \beta_1 = - \frac{a \alpha_1}{b} + \frac{\pi}{b}$, as in the proof of Lemma \ref{l:5}) belonging to $D^*$ holds $f_2 \left( \alpha_1, \beta_1 \right) < 0$ for $\left( \alpha_1, \beta_1 \right) \neq \left( 0, \frac{\pi}{b} \right)$ and $f_2 \left( 0, \frac{\pi}{b} \right) = 0$. Actually, we can restrict domain $D^*$ using $\beta_{\lambda} = - \frac{a \lambda}{b} + \frac{\pi}{b}$ instead of $\mu$. Then the function $\bar{g}_2 = (g_1^2, g_2^2)$ (introduced as in the proof of Lemma \ref{l:5}) on the domain $D^0 = \left\{ \left( \alpha_1, \beta_1 \right): \alpha_1 \in \left[ 0, \lambda \right] , \beta_1 \in \left[ \beta_{\lambda}, \frac{\pi}{b} \right] \right\}$ is mapping with contraction and therefore our assumption is correct.

If $b > a \geq 3$ as in the proof of Lemma \ref{l:5} on the part of the $d$ belonging to $D^*$ holds $f_1 \left( \alpha_1, \beta_1 \right) > 0$. That means the system (\ref{edge}) has no solutions on $D^*$.

In the case $a = 2$, $b \geq 4$ we have to consider only $\alpha_1 \geq \frac{\pi}{3} \geq \lambda$. Then the sign of $f_1$ will be considered on the straight line $\tilde{d}: \beta_1 = \frac{-6 \alpha_1}{b} + \frac{3 \pi}{b}$ which passes through the points $\left( \frac{\pi}{3}, \frac{\pi}{b} \right)$ and $\left( \frac{\pi}{2}, 0 \right)$ (and which is above the diagonal $d$). Then, on the domain $\tilde{D} = \left\{ \left( \alpha_1, \beta_1 \right): \alpha_1 \in \left[ \frac{\pi}{3}, \frac{\pi}{2} \right] , \beta_1 \in \left[ 0, \frac{\pi}{b} \right] \right\}$ the function $\bar{g}^0 = (g_1^0, g_2^0)$ where 
%%%%%%%
\begin{eqnarray*} 
g_1^0 & = & (1 - t_1) \alpha_1 + t_1 (- \frac{b \beta_1}{6} + \frac{\pi}{2}) + \frac{f_1}{k_1}       \\
g_2^0 & = & (1 - t_2) \beta_1 + t_2 ( \frac{-6 \alpha_1}{b} + \frac{3 \pi}{b} ) 
\end{eqnarray*} 
is mapping with contraction. That means the system $\alpha_1 = g_1^0$, $\beta_1 = g_2^0$, equivalent to the system $f_1 = 0$, $\beta_1 = \frac{-6 \alpha_1}{b} + \frac{3 \pi}{b}$ has unique solution $\left( \frac{\pi}{2}, 0 \right)$. So, on the line $\tilde{d}$ within domain $D^*$, holds $f_1 \left( \alpha_1, \beta_1 \right) > 0$. Again the system (\ref{edge}) has no solutions on $D^*$.
\end{proof}

%%%%%%%%%%%%%%%%%

As previous Lemmas \ref{l:5} - \ref{l:8} are showing that edge conditions are necessary, we shall note here that they are also sufficient. That is because they give us the possibility to calculate the values of the dihedral angles $\alpha_1$, $\alpha_2$, $\beta_1$, $\beta_2$, for a given values of the parameters $a$ and $b$, e.g. by the iteration method for nonlinear systems. The following Table gives the values of $\alpha_1$, $\beta_1$ (in radians) for the initial values of $a$, $b$. A more detailed table is provided in \cite{S95}, as well as a graphical representation for individual cases.

\vspace{5mm} 

\centerline{TABLE}

\vspace{3mm} 

\begin{tabular}{|c|c|l|l|c|c|c|l|l|} 
\cline{1-4} \cline{6-9}
$a$ & $b$ & $\alpha_1$ & $\beta_1$ &  & $a$ & $b$ & $\alpha_1$ & $\beta_1$ \\
\cline{1-4} \cline{6-9}
2 & 3 & 1.20394  & 0.5969756 & \hspace{4mm} & 4 & 9 & 0.775143 & 1.856247E-02 \\
2 & 4 & 1.332343 & 0.3618578 &  						& 5 & 6 & 0.4764179 & 0.293696 \\
2 & 5 & 1.422264 & 0.218217 &  							& 5 & 7 & 0.5210571 & 0.2031649 \\
2 & 6 & 1.486543 & 0.1215925 &  						& 5 & 8 & 0.5562545 & 0.1346927 \\
2 & 7 & 1.534341 & 5.222918E-02 &  					& 5 & 9 & 0.5845534 & 8.110417E-02 \\
3 & 4 & 0.810013 & 0.4270995 &  						& 5 & 10 & 0.6078031 & 3.799561E-02 \\
3 & 5 & 0.8954629 & 0.2600982 &  						& 5 & 11 & 0.6283055 & 2.171273E-03 \\
3 & 6 & 0.9588664 & 0.1473277 &  						& 6 & 7 & 0.3923899 & 0.2568823 \\
3 & 7 & 1.006904  & 6.615758E-02 &  				& 6 & 8 & 0.4266642 & 0.1871642 \\
3 & 8 & 1.047133 & 4.563809E-03 &  					& 6 & 9 & 0.454274  & 0.1326011 \\
4 & 5 & 0.6026575 & 0.3451357 &  						& 6 & 10 & 0.4769213 & 8.873422E-02 \\
4 & 6 & 0.6630102 & 0.2239677 &  						& 6 & 11 & 0.4958152 & 5.269175E-02 \\
4 & 7 & 0.7093756 & 0.1364688 &  						& 6 & 12 & 0.5119572 & 2.248575E-02 \\
\cline{6-9}
4 & 8 & 0.7457473 & 7.034869E-02 &  			 	\multicolumn{5}{c}{ } \\
\cline{1-4} 
\end{tabular}

%%%%%%%%%%%%%%%%%%%%%%%%%%%%55

\section{Summary}

On the base of the symmetries of the considered fundamental simplices, there are considered the 'edge conditions'. These conditions arise in situations when the dihedral angles around the edges from the same equivalence class are different. They restrict previously discovered possible cases of simplices that are realized in spaces of constant curvature. On the other hand this method gives us possibility to discover cases of fundamental simplices which are realized in other spaces of non-constant curvature. 

If we think about such cases in non maximal families, besides considered here from family F12 and previously considered in \cite{S93} from F7, F11, the candidates for such kind of examination are also simplices from families F8 and F10. In the maximal families there are also cases of simplices which are investigated here or in \cite{S93}, as well as the cases which have not yet been considered.

%%%%%%%%%%%%%%%%%%%%%%%%%%%%%%%%%%%
%%%%%%%%%%%%%%%%%%%%%%%

\end{document}